\documentclass[a4paper,11pt,reqno]{amsart}
\usepackage{graphicx, psfrag, amsmath,  amssymb, array, hyperref}

\newtheorem{lem}{Lemma}[section]
\newtheorem{thm}[lem]{Theorem}

\newtheorem{exa}[lem]{Example}

\numberwithin{equation}{section}

\newcommand{\ZZ}{{\mathbb{Z}}}
\newcommand{\RR}{{\mathbb{R}}}

\newcommand{\A}{{\mathcal{A}}}
\newcommand{\B}{{\mathcal{B}}}

\newcommand{\DDD}{{\mathcal{D}}}

\newcommand{\comp}{{\mathsf{comp}}}
\newcommand{\traj}{{\mathsf{traj}}}

\newcommand{\area}{{\textsf{area}}}

\newcommand{\N}{\mathsf{N}}
\newcommand{\E}{\mathsf{E}}
\newcommand{\U}{\mathsf{U}}
\newcommand{\HH}{\mathsf{H}}
\newcommand{\DD}{{\mathsf{D}}}

\newcommand{\G}{{\mathcal{G}}}

\newcommand{\F}{{\mathcal{F}}}
\newcommand{\M}{{\mathcal{M}}}

\begin{document}

\title[]{Meanders and Dyck-Path Billiards}\thanks{This research was supported in part by
National Science and Technology Council, Taiwan through grants 113-2115-M-003-010 (S.-P. Eu), 113-2115-M-153-002 (T.-S. Fu) and 112-2115-M-032-002 (H.-C. Hsu).}
\author[S.-P. Eu]{Sen-Peng Eu}
\address{Department of Mathematics, National Taiwan Normal University, Taipei 116325, and Chinese Air Force Academy, Kaohsiung 820009, Taiwan, ROC}
\email{speu@math.ntnu.edu.tw}

\author[T.-S. Fu]{Tung-Shan Fu}
\address{Department of Applied Mathematics, National Pingtung University, Pingtung 900391, Taiwan, ROC}
\email{tsfu@mail.nptu.edu.tw}

\author[H.-C. Hsu]{Hsiang-Chun Hsu}\address{Department of Mathematics, Tamkang University, New Taipei City 25137, Taiwan, ROC} \email{hchsu0222@gmail.com}

\begin{abstract}
We study a statistic $\traj$ on the ordered pairs $(P,Q)$ of Dyck paths of size $n$, which counts the number of billiard trajectories in the grid polygon enclosed by $P$ and $-Q$, where $-Q$ is the path obtained by reflecting $Q$ over the ground line.
In terms of grid polygon, we establish an involution on the set of such ordered pairs $(P,Q)$ which either increases or decreases $\traj(P,Q)$ by 1. This proves a result by Di Francesco--Golinelli--Guitter that the numbers of semimeanders (meanders, respectively) of order $n$ with even and odd numbers of components are equal if $n$ is even and differ by a Catalan number (the square of a Catalan number, respectively) if $n$ is odd. Some results about $(-1)$-evaluation of the generating functions for the statistic $\traj$ on restricted sets of Dyck paths are also presented.
\end{abstract}
\maketitle
\section{Introduction}
\subsection{Trajectory statistic for Dyck paths} 
A \emph{Dyck path} of size $n$ is a lattice path in the plane $\ZZ^2$ from $(0,0)$ to $(n,n)$ using \emph{north steps} $\N=(0,1)$ and \emph{east steps} $\E=(1,0)$ which never goes below the line $y=x$. Sometimes, we use the notation $\N^k$ ($\E^k$, respectively) for $k$ consecutive north steps (east steps, respectively). Let $\DDD_n$ denote the set of Dyck paths of size $n$. It is known that $|\DDD_{n}|=c_n=\frac{1}{n+1}\binom{2n}{n}$, the $n$th Catalan number. For any $P\in\DDD_n$, let $-P$ denote the lattice path obtained by reflecting the path $P$ over the line $y=x$, called a \emph{negative Dyck path}.

Let $P,Q\in\DDD_n$. We associate the ordered pair $(P,Q)$ with a \emph{grid polygon} in the plane $\ZZ^2$ enclosed by $P$ and $-Q$. The boundary of the grid polygon consists of the steps of $P$ and $-Q$, which are numbered $1, 2,\dots, 4n$ clockwise. For example, the grid polygon associated to $(P,Q)$, where $P=\N\N\E\N\N\E\N\E\E\E$ and $Q=\N\N\E\E\N\E\N\N\E\E$, is shown in Figure \ref{fig:meander-polygon}.

Motivated by the study of rainbow meanders and Cartesian billiards by Fiedler--Casta\~{n}eda \cite{FC} (see also \cite[Section~2.7]{KL}), we consider billiard trajectories, which are also interpreted as trajectories of light beams, in the grid polygon $(P,Q)$.  Pick some step $p$ on the boundary of the grid polygon, say with label $i$. Emit a  beam of light from the midpoint of $p$ into the interior of the polygon so that the light beam forms a 45$^\circ$ angle with $p$ and travels either northeast, southeast, southwest, or northwest (depending on the orientation of $p$). The light beam will travel through the interior of the polygon until reaching the midpoint of another step, with label $\pi(i)$, meeting it at a 45$^\circ$ angle. This defines a permutation $\pi: [4n]\rightarrow [4n]$.    As shown in Figure \ref{fig:meander-polygon}, the cycle decomposition of $\pi$ of that polygon is $\pi=(1,10,11,14,7,8,13,12,9,4,17,20)(2,3,18,19)(5,6,15,16)$. Each cycle of $\pi$ corresponds to a \emph{trajectory} of a beam of light (cf. \cite{FC, KL}).

\begin{figure}[ht]
\begin{center}
\includegraphics[width=2in]{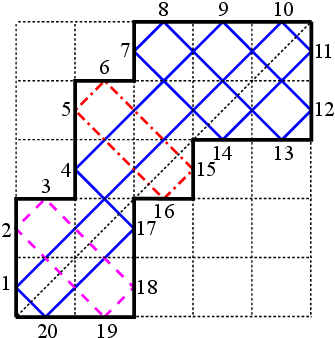}
\end{center}
\caption{\small A grid polygon with three trajectories of light beams.} \label{fig:meander-polygon}
\end{figure}

Let $\traj(P,Q)$ denote the number of trajectories of light beams in the grid polygon associated to $(P,Q)$. The initial values for the distribution of the ordered pairs $(P,Q)\in \DDD_n\times\DDD_n$ with respect to $\traj(P,Q)$ are shown in Table \ref{tab:meander-component-distribution}.

\begin{table}[ht]
\caption{\small The number of ordered pairs $(P,Q)\in\DDD_n\times\DDD_n$ with $k$ trajectories of light beams for $1\le k\le n\le 7$.}
\centering
\begin{tabular}{c|rrrrrrr}
 $n\backslash k$   & 1 &  2   &  3 & 4 & 5 & 6 & 7   \\
\hline
1 &   1 &     &      &       &       &       &   \\
2 &   2 &   2 &      &       &       &       &   \\
3 &   8 &  12 &   5  &       &       &       &   \\
4 &  42 &  84 &  56  &  14   &       &       &   \\
5 & 262 & 640 & 580  & 240   &   42  &       &   \\
6 & 1828 & 5236 & 5894 & 3344 & 990  & 132   &   \\
7 & 13820 & 45164 & 60312 & 42840 & 17472 & 4004 &  429  
\end{tabular}
\label{tab:meander-component-distribution}
\end{table}

In particular, we define the statistic $\traj$ on Dyck paths as follows. Let $L_n$ denote the Dyck path $\N^n\E^n$. For any $P\in\DDD_n$, let $\traj(P)=\traj(P,L_n)$ for short. For $n=3$, the values of $\traj(P)$ for the five Dyck paths $P$ of size 3 shown in Figure  \ref{fig:Dyck-path-size-3} are 3, 2, 1, 1, and 2 accordingly. 
The initial values for the distribution of Dyck paths with respect to the statistic $\traj$ are shown in Table \ref{tab:semimeander-component-distribution}.

\begin{figure}[ht]
\begin{center}
\includegraphics[width=5.6in]{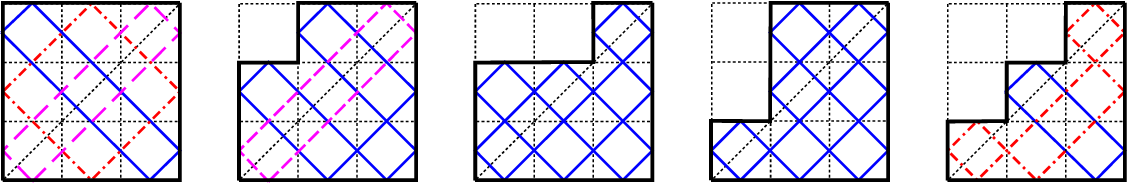}
\end{center}
\caption{\small The trajectories of light beams for the  paths in $\DDD_3$.} \label{fig:Dyck-path-size-3}
\end{figure}

\begin{table}[ht]
\caption{\small The number of Dyck path $P\in\DDD_n$ with $\traj(P)=k$ for $1\le k\le n\le 7$.}
\centering
\begin{tabular}{c|rrrrrrr}
 $n\backslash k$   & 1 &  2   &  3 & 4 & 5 & 6 & 7   \\
\hline
1 &   1 &     &      &      &       &       &   \\
2 &   1 &   1 &      &      &       &       &   \\
3 &   2 &   2 &   1  &      &       &       &   \\
4 &   4 &   6 &   3  &  1   &       &       &   \\
5 &  10 &  16 &  11  &  4   &   1   &       &   \\
6 &  24 &  48 &  37  & 17   &   5   &   1   &   \\
7 &  66 & 140 & 126  & 66   &  24   &   6   &  1  
\end{tabular}
\label{tab:semimeander-component-distribution}
\end{table}

\subsection{Meanders/Semimeanders} Fix a straight line in the plane $\RR^2$ and $2n$ points on it. A \emph{meander} of order $n$ is a collection of closed self-avoiding and noncrossing curves that intersect the line in exactly those $2n$ points. Two configurations are considered as equivalent if they are smooth deformations of one another. We consider only the inequivalent meanders. For example, the four meanders of order 2 are shown in Figure \ref{fig:meander-order-2}. Let $\M_n$ denote the set of meanders of order $n$. Let $\comp(T)$ denote the number of components of a meander $T\in\M_n$. Note that the triangular array in Table \ref{tab:meander-component-distribution} coincides with the distribution of the meanders in $\M_n$ with respect to the number of components \cite[A008828]{oeis}; see also \cite{DiFGG}. We refer the readers to \cite{Zvonkin} for historical backgrounds and recent developments of the enumeration of meanders.

\begin{figure}[ht]
\begin{center}
\includegraphics[width=4in]{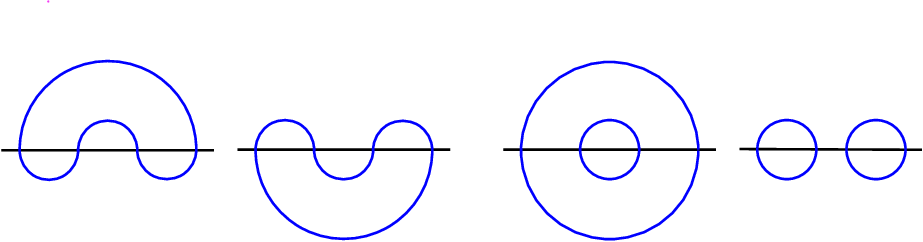}
\end{center}
\caption{\small The four meanders of order 2.} \label{fig:meander-order-2}
\end{figure}

Now, fix a ray in the plane $\RR^2$ and $n$ points on it. A \emph{semimeander} of order $n$ is a collection of closed self-avoiding and noncrossing curves that intersect the ray in exactly those $n$ points. 
For example, the five semimeanders of order 3 are shown in Figure \ref{fig:semimeander-order-3}.
Let $\M^*_n$ denote the set of semimeanders of order $n$. Note that the triangular array in Table \ref{tab:semimeander-component-distribution} coincides with the distribution of the semimeanders in $\M^*_n$ with respect to the number of components \cite[A046726]{oeis}; see also \cite{DiFGG}. Recently, Cori and Hetyei \cite{CH} present a bijection between meanders (semimeanders, respectively) of order $n$ and specialized hypermaps, called the spanning hypertrees of the reciprocal of a dipole with $n$ parallel edges (a monopole with $n/2$ nested edges, respectively).

\begin{figure}[ht]
\begin{center}
\includegraphics[width=4.85in]{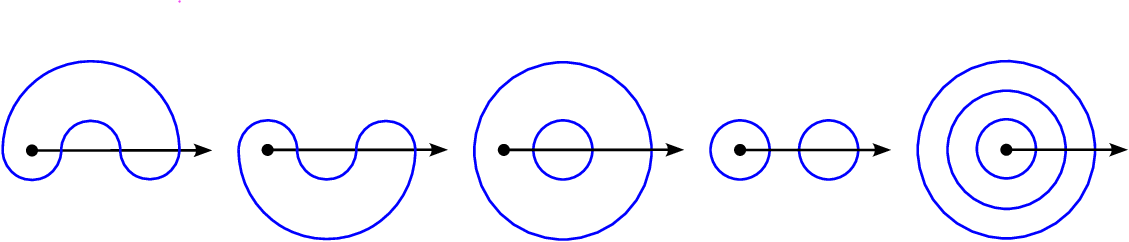}
\end{center}
\caption{\small The five semimeanders of order 3.} \label{fig:semimeander-order-3}
\end{figure}

\subsection{Main results}
There is a standard bijection between ordered pairs of Dyck paths and meanders; see \cite{DiFGG-NPB, LaCroix}.
Our first result is to describe the bijective properties in Theorem \ref{thm:Dyck-path-to-semimeander}

\begin{thm} \label{thm:Dyck-path-to-semimeander} The following results hold.
\begin{enumerate}
\item There is a bijection $(P,Q)\mapsto T$ of $\DDD_n\times\DDD_n$ onto $\M_n$ with $\traj(P,Q)=\comp(T)$.
\item There is a bijection $P\mapsto S$ of $\DDD_n$ onto $\M^*_n$ with $\traj(P)=\comp(S)$.
\end{enumerate}
\end{thm}

Our second result is to give a combinatorial proof of the  $(-1)$-evaluations of the generating functions for the statistic $\traj$ on Dyck paths and ordered pairs of Dyck paths.

\begin{thm} \label{thm:traj-balance} The following results hold.
\begin{enumerate}
\item There is an involution $\varphi:P\mapsto P'$ on $\DDD_n$ satisfying  $\traj(P)-\traj(P')\in\{1,0,-1\}$ and resulting in the following identity
\begin{equation} 
\sum_{P\in\DDD_n} (-1)^{\traj(P)} = 
\begin{cases} 
0 & \mbox{if $n$ is even,} \\
-c_k & \mbox{if $n=2k+1$.}
\end{cases}
\end{equation}
\item There is an involution $\psi:(P,Q)\mapsto (P',Q')$ on $\DDD_n\times\DDD_n$ satisfying  $\traj(P,Q)-\traj(P',Q')\in\{1,0,-1\}$ and resulting in the following identity
\begin{equation} 
\sum_{(P,Q)\in\DDD_n\times\DDD_n} (-1)^{\traj(P,Q)} = 
\begin{cases} 
0 & \mbox{if $n$ is even,} \\
-c_k^2 & \mbox{if $n=2k+1$.}
\end{cases}
\end{equation}
\end{enumerate}
\end{thm}

By Theorems \ref{thm:Dyck-path-to-semimeander} and \ref{thm:traj-balance}, we prove the following result \cite[Eq.\,(6.1)]{DiFGG} anew.

\begin{thm}[Di Francesco--Golinelli--Guitter] \label{thm:comp-balance} We have
\begin{enumerate}
\item ${\displaystyle
\sum_{S\in\M^*_n} (-1)^{\comp(S)} = 
\begin{cases} 
0 & \mbox{if $n$ is even} \\
-c_k & \mbox{if $n=2k+1$,}
\end{cases}
}$
\item ${\displaystyle
\sum_{T\in\M_n} (-1)^{\comp(T)} = 
\begin{cases} 
0 & \mbox{if $n$ is even} \\
-c_k^2 & \mbox{if $n=2k+1$.}
\end{cases}
}$
\end{enumerate}
\end{thm}

A \emph{Motzkin path} of length $n$ is a lattice path in the plane $\ZZ^2$ from $(0,0)$ to the point $(n,0)$ using \emph{up step} $(1,1)$, \emph{horizontal step} $(1,0)$ and \emph{down step} $(1,-1)$ which never goes below the $x$-axis. Motzkin paths are counted by the \emph{Motzkin numbers} $\{m_n\}_{n\ge 1}=\{1, 2, 4, 9, 21, 51, 127$, $323, \dots \}$  \cite[A001006]{oeis}. Those Motzkin paths with no horizontal step on the $x$-axis are called \emph{Riordan paths}, which are counted by the \emph{Riordan numbers} $\{r_n\}_{n\ge 1}=\{0, 1, 1, 3, 6, 15, 36, \dots\}$ \cite[A005043]{oeis}. For convenience, a Motzkin path or  Riordan path is said to be \emph{even} (\emph{odd}, respectively) if it contains an even (odd, respectively) number of up steps.

A \emph{peak} of a Dyck path is an occurrence of $\N\E$ (i.e., a north step followed by an east step) and a \emph{valley} is an occurrence of $\E\N$. The coordinates of a peak or valley is given by the intersection point of its $\N$ and $\E$. A peak is said to be at \emph{height} $h$ if the intersection point is on the line $y=x+h$. It is known that the Dyck paths of size $n$ with all peaks at odd (even, respectively) height are counted by the Motzkin number $m_{n-1}$ (Riordan number $r_{n}$, respectively) \cite{Callan}.
Our third result is about the $(-1)$-evaluation of the generating functions for the statistic $\traj$ on restricted sets of Dyck paths, which coincide up to a sign with the excess of the number of even Motzkin/Riordan paths over the odd ones (cf. \cite[A343773]{oeis}).

\newpage
\begin{thm} \label{thm:up-to-a-sign} The following results hold.
\begin{enumerate}
\item Let $\A_n\subset\DDD_n$ be the set of Dyck paths with all peaks at odd height. Then we have
\begin{equation} \label{eqn:A_n+1}
\sum_{P\in\A_{n+1}} (-1)^{\traj(P)}=(-1)^{\lfloor\frac{n+2}{2} \rfloor} (e_n-d_n),
\end{equation}
where $e_n$ ($d_n$, respectively) is the number of even (odd, respectively) Motzkin paths of length $n$.
\item Let $\B_n\subset\DDD_n$ be the set of Dyck paths with all peaks at even height. Then we have
\begin{equation} \label{eqn:B_n}
\sum_{P\in\B_n} (-1)^{\traj(P)}=(-1)^{\lfloor\frac{n+2}{2} \rfloor} (\hat{d}_n-\hat{e}_n),
\end{equation}
where $\hat{e}_n$ ($\hat{d}_n$, respectively) is the number of even (odd, respectively) Riordan paths of length $n$.
\end{enumerate}
\end{thm}


\section{Bijective properties}

\subsection{Arch configurations}

A \emph{noncrossing matching} of order $n$ is a noncrossing partition of the set $\{1,\dots, 2n\}$ into blocks of size two; see \cite{Simion}. For any meander $T\in\M_n$, the straight line divides the plane into two half-planes. If the $2n$ points on the straight line are numbered from 1 to $2n$, the part of $T$ in each half-plane determines a noncrossing matching of order $n$. Such a system of $n$ arches in the upper (lower, respectively) half-plane is called an \emph{arch configuration} of order $n$. For example, the meander shown in Figure \ref{fig:arch-configuration} is the superimposition of two arch configurations, the upper (lower, respectively) configuration of which is given by the noncrossing matching $\{(1,10),(2,5),(3,4),(6,9),(7,8)\}$ ($\{(1,4),(2,3),(5,6),(7,10),(8,9)\}$, respectively).

\begin{figure}[ht]
\begin{center}
\includegraphics[width=2.75in]{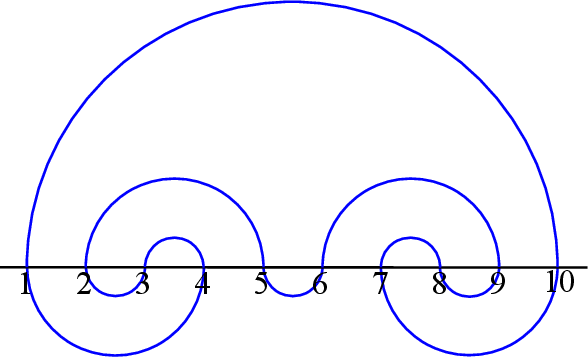}
\end{center}
\caption{\small A meander is the superimposition of two arch configurations.} \label{fig:arch-configuration}
\end{figure}  

Semimeanders of order $n$ can be viewed as the specialization of meanders of order $n$ to those with the prescribed noncrossing matching $\{(1,2n),(2n,2n-1),\dots,(n,n+1)\}$ as their lower arch configuration, which are also called \emph{rainbow meanders} \cite{DiFGG}. The straight line of the meanders is interpreted as a pair of rays originating from the center of the lower arch configuration. Fold the pair of rays into one, as indicated in Figure \ref{fig:semimeander-fold}, so that those $2n$ points on the straight line are identified by pairs according to the lower arch configuration. This provides a transformation to obtain the semimeander from a specialized meander; see \cite{DiFGG, LaCroix}.

\begin{figure}[ht]
\begin{center}
\includegraphics[width=4.2in]{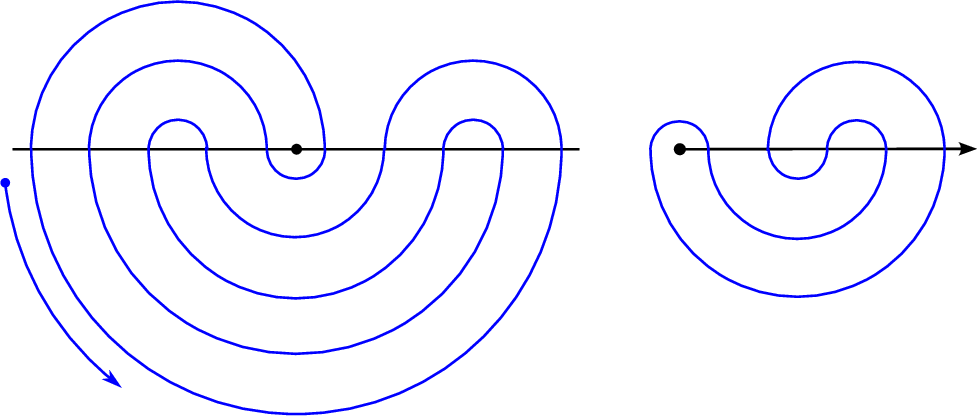}
\end{center}
\caption{\small A specialized meander is folded into a semimeander.} \label{fig:semimeander-fold}
\end{figure}

\subsection{Bijective properties}

A north or east step of a Dyck path $P$ is said to be at \emph{height} $h$ if it is between the two lines $y=x+h-1$ and $y=x+h$.  
We say that a north step $p$ and an east step $q$ form a \emph{matching pair} of $P$ if $p$ and $q$ are at the same height and face each other in the sense that the line connecting the midpoints of $p$ and $q$ stays under the path. We identify the matching pairs of $-P$ with the matching pairs of $P$.

In the following, we describe the bijection between $\DDD_n\times\DDD_n$ and $\M_n$ in terms of arch configuration.

\medskip
\noindent
\emph{Proof of Theorem \ref{thm:traj-balance}.}
Given an ordered pair $(P,Q)\in\DDD_n\times\DDD_n$, let $P=p_1p_2\cdots p_{2n}$ and let $Q=q_1q_2\cdots q_{2n}$. The corresponding meander $T\in\M_n$ is constructed as follows.
\begin{enumerate}
\item Draw a straight line in the plane $\RR^2$ and $2n$ points on it, numbered from 1 to $2n$.
\item  Find all the matching pairs of $P$. For each matching pair, say $(p_i,p_j)$, draw an arch from $i$ to $j$ above the points.
\item Find all the matching pairs of $Q$. For each matching pair, say $(q_i,q_j)$, draw an arch from $i$ to $j$ below the points.
\end{enumerate}

The inverse map can be established in reverse. To describe the billiard trajectories in the grid polygon, let $z_1,z_2,\dots, z_{2n}$ be the $2n$ points on the line $y=x$ in the grid polygon such that $z_k$ is the intersection of the two lines $y=x$ and $x+y=\frac{2k-1}{2}$, i.e., the coordinate of $z_k$ is $(\frac{2k-1}{4},\frac{2k-1}{4})$. 
\begin{enumerate}
\item For each upper arch of $T$, say from $i$ to $j$,  draw a line segment connecting $z_i$ to the midpoint of $p_i$, draw a line segment connecting the midpoints of $p_i$ and $p_j$, and then draw a line segment from the midpoint of $p_j$ to $z_j$.
\item For each lower arch of $T$, say from $i$ to $j$,  draw a line segment connecting $z_i$ to the midpoint of $q_i$, draw a line segment connecting the midpoints of $q_i$ and $q_j$, and then draw a line segment from the midpoint of $q_j$ to $z_j$.
\end{enumerate}
Thus, the billiard trajectories are deformations of the components of $T$. 
The result of Theorem \ref{thm:Dyck-path-to-semimeander}(i) follows.

To establish the bijection $\DDD_n\rightarrow\M^*_n$, for any $P\in\DDD_n$ we construct the meander corresponding to $(P,L_n)$, whose lower arch configuration is the noncrossing matching $\{(1,2n)$, $(2,2n-1),\dots,(n,n+1)\}$. 
The proof of Theorem \ref{thm:Dyck-path-to-semimeander} is completed.
\qed

\medskip
\begin{exa} {\rm
Let $P=\N\N\N\E\E\N\N\E\E\E$ and $Q=\N\N\E\E\N\E\N\N\E\E$. The grid polygon associated to $(P,Q)$ is shown in Figure \ref{fig:meander-2-upper}. Let $p_1,p_2,\dots, p_{10}$ denote the steps of $P$. The matching pairs of $P$ are $(p_1,p_{10}), (p_2,p_5), (p_3,p_4), (p_6,p_9)$ and $(p_7,p_8)$. This determines the upper arch configuration of the  the meander in Figure \ref{fig:arch-configuration}. Let $q_1,q_2,\dots,q_{10}$ denote the steps of $Q$. The matching pairs of $Q$ are $(q_1,q_4), (q_2,q_3), (q_5,q_6), (q_7,q_{10})$ and $(q_8,q_9)$. This determines the lower arch configuration of the  the meander in Figure \ref{fig:arch-configuration}. Note that the unique billiard trajectory shown in Figure \ref{fig:meander-2-upper} is a deformation of the meander in Figure \ref{fig:arch-configuration}.
}
\end{exa}

\begin{figure}[ht]
\begin{center}
\psfrag{1}[][][0.85]{$z_1$}
\psfrag{2}[][][0.85]{$z_2$}
\psfrag{3}[][][0.85]{$z_3$}
\psfrag{4}[][][0.85]{$z_4$}
\psfrag{5}[][][0.85]{$z_5$}
\psfrag{6}[][][0.85]{$z_6$}
\psfrag{7}[][][0.85]{$z_7$}
\psfrag{8}[][][0.85]{$z_8$}
\psfrag{9}[][][0.85]{$z_9$}
\psfrag{10}[][][0.85]{$z_{10}$}
\includegraphics[width=2in]{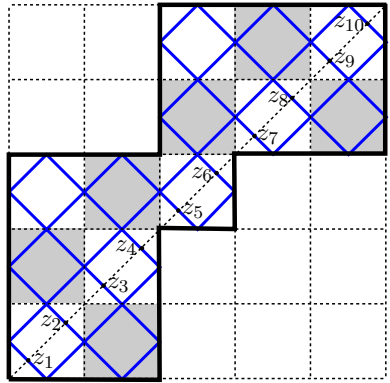}
\end{center}
\caption{\small A meander is the superimposition of two arch configurations.} \label{fig:meander-2-upper}
\end{figure}  

\section{Proof of Theorem \ref{thm:traj-balance}}

For any $(P,Q)\in\DDD_n\times\DDD_n$, a trajectory of light beam in the grid polygon associated to $(P,Q)$ can be partitioned into two families of parallel line segments, on the lines with slope 1 and $-1$, respectively. If a trajectory is assigned an orientation, we observe that the parallel segments of light of the trajectory, on different lines with the same slope, move alternately in the opposite directions.

\begin{lem} \label{lem:alternating-direction} Let $\ell_1,\ell_2$ be two parallel segments of light in a trajectory. The following properties hold.
\begin{enumerate}
\item
If $\ell_1,\ell_2$ are on the lines $y=x+\frac{4k+1}{2}$ and $y=x+\frac{4k+3}{2}$, respectively, for some integer $k$ then $\ell_1, \ell_2$ move in the opposite directions.
\item 
If $\ell_1,\ell_2$ are on the lines $y=-x+\frac{4k+1}{2}$ and $y=-x+\frac{4k+3}{2}$, respectively, for some integer $k$ then $\ell_1, \ell_2$ move in the opposite directions.
\end{enumerate}
\end{lem}

\begin{proof}
For convenience, we color the unit squares in the grid polygon in black and white like
a checkerboard. A unit square is colored white if its upper right corner $(i, j)$ satisfies the
condition that $i+j$ is even, and black otherwise, as shown in Figure \ref{fig:meander-2-upper}. Choose a white square along the path $P$ and emit a beam of light from the midpoint of its left edge to the midpoint of its top edge. We have the following observations.

\begin{enumerate}
\item The light beam goes northeast-bound, hits an east step of $P$, which is the top edge of a white square, and turns 90$^\circ$ clockwise. 
\item Then the light beam either hits a north step (the right edge of a white square) of $-Q$, turns 90$^\circ$ clockwise and hits its matching east step (the bottom edge of a white square), or hits an east step (the bottom edge of a black square) of $-Q$,  turns 90$^\circ$ counter clockwise and hits its matching north step (the right edge of a black square).
\item In the former case, the light beam turns 90$^\circ$  clockwise and hits either a north step or an east step of $P$. 
In the latter case, the light beam turns 90$^\circ$ counter clockwise and hits either an east step or a north step of $P$.
\end{enumerate} 
Along the resulting trajectory, we observe that the segments of light within each white (black, respectively) square move clockwise (counter clockwise, respectively). The result follows.
\end{proof}

\begin{exa} {\rm
For the grid polygon shown in Figure \ref{fig:meander-polygon}, consider the trajectory of the cycle $(1,10,11,14,7,8,13,12,9,4,17,20)$ of that permutation $\pi$. If the trajectory is oriented in the direction from the entry 20 to the entry 1, then the light segments between the entries $1$ and $10$, between the entries $7$ and $8$, and between the entries $13$ and $12$ move in northeast direction. Moreover, the light segments between the entries $11$ and $14$, between the entries $9$ and $4$, and between the entries $17$ and $20$ move in southwest direction. 
}
\end{exa}

\begin{lem} \label{lem:peak-valley}
For $n\ge 2$, if $P, P'$ are two Dyck paths in $\DDD_n$ such that $P'$ is obtained from $P$ by replacing a peak  by a valley then $\traj(P)-\traj(P')=\pm 1$. 
\end{lem}

\begin{proof}
Let $P=p_1p_2\cdots p_{2n}$. Suppose $p_ip_{i+1}=\N\E$ is a peak of $P$ at height at least 2 and 
$P'$ is obtained from $P$ by replacing $p_ip_{i+1}$ by $\E\N$.  Depicted in Figures \ref{fig:single-trajectory}(i) and \ref{fig:two-trajectories}(i), let $b,c$ be the midpoints of $p_i, p_{i+1}$, respectively. Then the shaded unit squares are enclosed in the grid polygon associated to $(P,L_n)$, and the four points $\{a, b, c, d\}$ ($\{e, f, g, h\}$, respectively) are in a trajectory of a beam of light. There are two cases.

Case 1. Those eight points are in the same trajectory. By Lemma \ref{lem:alternating-direction},  if the trajectory is oriented from $b$ to $c$ then the light moves in the direction from $g$ to $f$. The relative order of those eight points is shown in Figure \ref{fig:single-trajectory}(ii). If the peak $p_ip_{i+1}$ is replaced by a valley, as shown in Figure \ref{fig:single-trajectory}(iii)-(iv), then the trajectory is split into two trajectories, one passing the points $a, f$ and $e$ and the other passing the points $h, g$ and $d$. Thus, $\traj(P)-\traj(P')=-1$.

Case 2. Those eight points are in two different trajectories. The relative order of those eight points is shown in Figure \ref{fig:two-trajectories}(ii). If the peak $p_ip_{i+1}$ is replaced by a valley, as shown in Figure \ref{fig:two-trajectories}(iii)-(iv), then these two trajectories are merged into one trajectory, passing the points $a, f, e, h, g$ and $d$. Thus, $\traj(P)-\traj(P')=1$.

The result follows.
\end{proof}

\begin{figure}[ht]
\begin{center}
\includegraphics[width=3in]{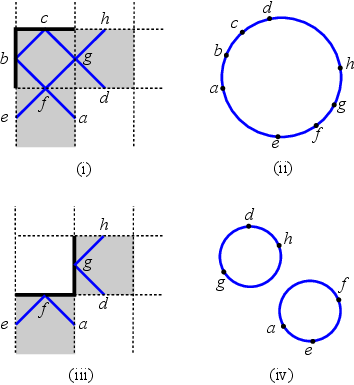}
\end{center}
\caption{\small A trajectory of light beam split into two trajectories.} \label{fig:single-trajectory}
\end{figure}

\begin{figure}[ht]
\begin{center}
\includegraphics[width=3in]{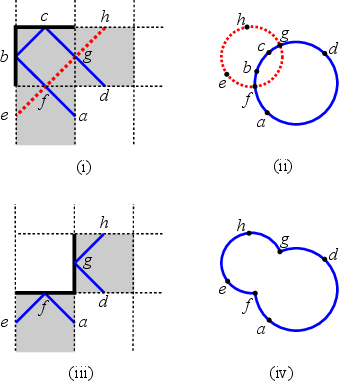}
\end{center}
\caption{\small Two trajectories of light beams merged into one trajectory.} \label{fig:two-trajectories}
\end{figure}

\begin{lem} \label{lem:valley-peak}
For $n\ge 2$, if $P, P'$ are two Dyck paths in $\DDD_n$ such that $P'$ is obtained from $P$ by replacing a valley by a peak  then $\traj(P)-\traj(P')=\pm 1$. 
\end{lem}

\begin{proof}
Let $P=p_1p_2\cdots p_{2n}$. Suppose $p_ip_{i+1}=\E\N$ is a valley of $P$ and 
$P'$ is obtained from $P$ by replacing $p_ip_{i+1}$ by $\N\E$. Depicted in Figures \ref{fig:single-trajectory}(iii) and \ref{fig:two-trajectories}(iii), let $f,g$ be the midpoints of $p_i, p_{i+1}$, respectively. Then the shaded unit squares are enclosed in the grid polygon $(P,L_n)$, and the three points $\{a, f, e\}$ ($\{d, g, h\}$, respectively) are in a trajectory of a beam of light. There are two cases.

Case 1. Those six points are in the same trajectory. By Lemma \ref{lem:alternating-direction},  if the trajectory is oriented from $e$ to $f$ then the light moves in the direction from $g$ to $h$. The relative order of those six points is shown in Figure \ref{fig:two-trajectories}(iv). If the valley $p_ip_{i+1}$ is replaced by a peak, as shown in Figure \ref{fig:two-trajectories}(i)-(ii), then the trajectory is split into two trajectories, one passing the points $e, f, g$ and $h$ and the other passing the points $a, b, c$ and $d$. Thus, $\traj(P)-\traj(P')=-1$.

Case 2. Those six points are in two different trajectories. The relative order of those six points is shown in Figure \ref{fig:single-trajectory}(iv). If the valley $p_ip_{i+1}$ is replaced by a peak, as shown in Figure \ref{fig:single-trajectory}(i)-(ii), then these two trajectories are merged into one trajectory, passing the points $a, b, c, d, h, g, f$ and $e$. Thus, $\traj(P)-\traj(P')=1$.

The result follows.
\end{proof}

In the following, we construct the involution $\varphi:\DDD_n\rightarrow\DDD_n$ in Theorem \ref{thm:traj-balance}(i).
First, we describe the fixed points of the involution. Let $\F_n\subset\DDD_n$ be the set of Dyck paths containing no peak with even $y$-coordinate and no valley with odd $x$-coordinate. For example, the grid polygon shown in Figure \ref{fig:Dyck-path-fixed-point} is associated to an ordered pair $(P,Q)\in\F_n\times\F_n$. (Note that the peaks (valleys, respectively) of $Q$ become the valleys (peaks, respectively) of the negative Dyck path $-Q$.)

\begin{figure}[ht]
\begin{center}
\includegraphics[width=2in]{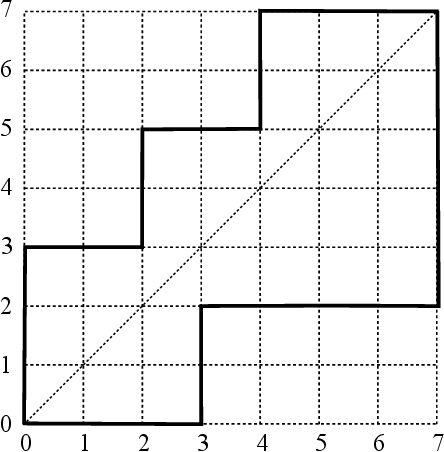}
\end{center}
\caption{\small A grid polygon associated to an ordered pair in $\F_n\times\F_n$.} \label{fig:Dyck-path-fixed-point}
\end{figure}

\begin{lem} \label{lem:number-of-fixed-points} For $n\ge 1$, we have
\begin{equation}
|\F_n|= 
\begin{cases} 
0 & \mbox{if $n$ is even,} \\
c_k & \mbox{if $n=2k+1$.}
\end{cases}
\end{equation}
\end{lem}

\begin{proof}
Note that $|\F_1|=1$. Let $n\ge 2$. We observe that for some $t\ge 1$, every path $P\in\F_n$ is of the form
\begin{equation} \label{eqn:Dyck-path-double}
P=\N^{2a_1+1}\E^{2b_1}\N^{2a_2}\E^{2b_2}\cdots\N^{2a_t}\E^{2b_t+1}
\end{equation}
for some positive integers $a_1,\dots,a_t$ and $b_1,\dots,b_t$ with $a_1+\cdots+a_t=b_1+\cdots+b_t=n-1$. Thus, every path in $\F_n$ has odd size, say $n=2k+1$. Moreover, the paths $P$ in (\ref{eqn:Dyck-path-double}) are in one-to-one correspondence to the paths $P'=\N^{a_1}\E^{b_1}\N^{a_2}\E^{b_2}\cdots\N^{a_t}\E^{b_t}$ in $\DDD_k$. The result follows.
\end{proof}

\medskip
\begin{lem} \label{lem:cyc-odd-number}
For all $(P,Q)\in\F_{2k+1}\times\F_{2k+1}$, $\traj(P,Q)$ is odd.
\end{lem}

\begin{proof} 
Let $L=\N^{2k+1}\E^{2k+1}$. Note that $L\in\F_{2k+1}$ and $\traj(L,L)=2k+1$.
For any $(P,Q)\in\F_{2k+1}\times\F_{2k+1}$, we define a serial operation to remove a $2\times 2$ square from the grid polygon associated to $(P,Q)$ while preserving the parity of $\traj$. Suppose $P$ contains a peak at height at least 4. Starting from such a peak, we create four Dyck paths $P_0=P, P_1, P_2, P_3, P_4$ such that $P_i$ is obtained from $P_{i-1}$ by replacing a peak by a valley, by the serial replacements shown in Figure \ref{fig:remove-four-peaks}. Note that $P_4$ is in $\F_{2k+1}$. By Lemma \ref{lem:peak-valley}, we have $\traj(P_4,Q)\equiv\traj(P,Q)\pmod 2$. (A symmetric operation can  apply to the path $Q$.) Moreover, the grid polygon of $(P,Q)$ can be obtained from that of $(L,L)$ by a sequence of such operations. Thus, $\traj(P,Q)\equiv\traj(L,L)\pmod 2$, which is odd. The result follows.
\end{proof}

\begin{figure}[ht]
\begin{center}
\includegraphics[width=4.2in]{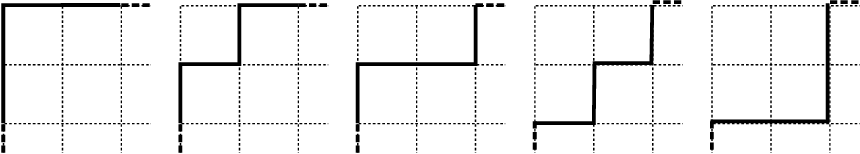}
\end{center}
\caption{\small A serial operation to remove a $2\times 2$ square from a grid polygon.} \label{fig:remove-four-peaks}
\end{figure}

For any Dyck path $P$, we say that a peak is \emph{good} (\emph{bad}, respectively) if its $y$-coordinate is odd (even, respectively), and that a valley is \emph{good} (\emph{bad}, respectively) if its $x$-coordinate is even (odd, respectively).  By a \emph{corner} of a Dyck path we mean a peak or a valley. If $P$ is of even size, we observe  that the last peak of $P$ is always a bad corner. 

\medskip
Now, we describe a simple construction of the involution $\varphi:P\mapsto P'$ on $\DDD_n$.
Let $P=p_1p_2\cdots p_{2n}$. If $P$ contains no bad corner then $n$ is odd and we set $\varphi(P)=P$. Otherwise, find the first bad corner in $P$, say $p_ip_{i+1}$.
The corresponding path $P'=p_1\cdots p_{i-1}p'_ip'_{i+1}p_{i+2}\cdots p_{2n}$ is obtained from $P$ by interchanging   $p_i, p_{i+1}$, i.e., $p'_ip'_{i+1}=p_{i+1}p_i$.

\begin{lem} \label{lem:involution} The map $\varphi$ is an involution. If $P\neq P'$ then $\traj(P)-\traj(P')= \pm 1$.
\end{lem}

\begin{proof} Suppose $P\neq P'$. Let $p_ip_{i+1}$ be the first bad corner of $P$. Then in the prefix $p_1\cdots p_i$ of $P$, the $x$-coordinate and $y$-coordinate of each corner $p_jp_{j+1}$ ($j<i$) have the opposite parities. There are two cases.

Case 1. $p_ip_{i+1}$ is a peak. Then both of its $x$-coordinate and $y$-coordinate are even. Thus, $p'_ip'_{i+1}$ is a valley with odd $x$-coordinate and $y$-coordinate.

Case 2. $p_ip_{i+1}$ is a valley. Then both of its $x$-coordinate and $y$-coordinate are odd. Thus, $p'_ip'_{i+1}$ is a peak with even $x$-coordinate and $y$-coordinate.

Note that in either case $p'_ip'_{i+1}$ is the first bad corner of $P'$. Thus, the map $\varphi$ is an involution. By Lemmas \ref{lem:peak-valley} and  \ref{lem:valley-peak}, we have $\traj(P)-\traj(P')= \pm 1$.
\end{proof}

\medskip
\noindent
\emph{Proof of Theorem \ref{thm:traj-balance}.}
By Lemmas \ref{lem:number-of-fixed-points}, \ref{lem:cyc-odd-number} and \ref{lem:involution}, we have
\begin{align*}
\sum_{P\in\DDD_n} (-1)^{\traj(P)} &= \sum_{P\in\F_n} (-1)^{\traj(P)} \\
&= -|\F_n|  \\
&=\begin{cases} 
0 & \mbox{if $n$ is even,} \\
-c_k & \mbox{if $n=2k+1$.}
\end{cases}
\end{align*}
The assertion (i) of Theorem \ref{thm:traj-balance} follows.

Now, we describe the involution $\psi:(P,Q)\mapsto (P', Q')$ on $\DDD_n\times\DDD_n$. If $(P,Q)$ is an ordered pair in $\F_n\times\F_n$ then $n$ is odd and we set $\psi((P,Q))=(P,Q)$. Otherwise, at least one of $P$ and $Q$ contains a bad corner. The corresponding ordered pair $(P',Q')$ is constructed as follows. If $P$ contains a bad corner, find the corresponding path $\varphi(P)$ and set the ordered pair $(P',Q')=(\varphi(P),Q)$. Otherwise, $P\in\F_n$ and $Q$ contains a bad corner. Find the corresponding path $\varphi(Q)$ and set the ordered pair $(P',Q')=(P,\varphi(Q))$.

By  Lemmas \ref{lem:number-of-fixed-points}, \ref{lem:cyc-odd-number} and \ref{lem:involution}, we have
\begin{align*}
\sum_{(P,Q)\in\DDD_n\times\DDD_n} (-1)^{\traj(P,Q)} &= \sum_{(P,Q)\in\F_n\times\F_n} (-1)^{\traj(P,Q)} \\
&= -|\F_n\times\F_n|  \\
&=\begin{cases} 
0 & \mbox{if $n$ is even,} \\
-c_k^2 & \mbox{if $n=2k+1$.}
\end{cases}
\end{align*}
The assertion (ii) of Theorem \ref{thm:traj-balance} follows.
\qed

\section{Proof of Theorem \ref{thm:up-to-a-sign}}
In this section, we study the $(-1)$-evaluation of the generating function for $\traj$ on the set $\A_{n+1}$ ($\B_n$, respectively) of Dyck paths with all peaks at odd (even, respectively) heights.

For any $P\in\DDD_n$, let $\area(P)$ denote the number of unit squares under the Dyck path $P$ and above the line $y=x$. For example, on the left hand side of Figure \ref{fig:Dyck-odd-Motzkin} is a Dyck path $P$ in $\G_7$ with $\area(P)=14$. 
In particular, we have $\area(L_n)=\binom{n}{2}$, where $L_n=\N^n\E^n$. Those $\binom{n}{2}$ unit squares are partitioned into $n-1$ rows such that the $i$th row from top has $n-i$ squares. 

\begin{figure}[ht]
\begin{center}
\includegraphics[width=3.6in]{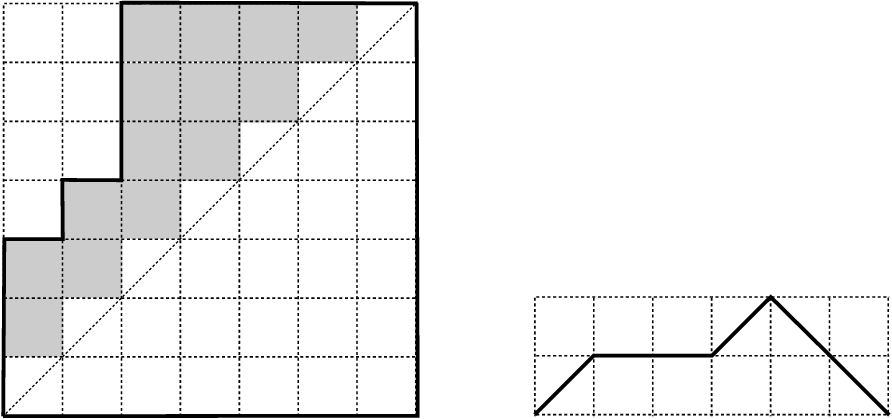}
\end{center}
\caption{\small A Dyck path $P\in\A_7$ and its corresponding Motzkin path $\phi_A(P)$.} 
\label{fig:Dyck-odd-Motzkin}
\end{figure}

Let $Z_n=(\N\E)^n$. Note that $Z_n$ is a Dyck path in $\A_n$ with $\area(Z_n)=0$. Using Lemmas \ref{lem:peak-valley} and \ref{lem:valley-peak}, we determine the parity of $\traj(P)$ from the values $\area(Z_n)$ and $\traj(Z_n)$.

\begin{lem} \label{lem:Z_n+1} For $n\ge 1$, $\traj(Z_{n+1})$ is odd if $n\equiv 0, 1\pmod 4$ and even if $n\equiv 2, 3\pmod 4$.
\end{lem}

\begin{proof} Let $T$ be the meander corresponding to the ordered pair $(Z_{n+1},L_{n+1})$. Note that the upper (lower, respectively) arch configuration of $T$ is the noncrossing matching $\{(1,2), (3,4)$, $\dots, (2n+1,2n+2)\}$ ($\{(1,2n+2), (2,2n+1), \dots, (n+1,n+2)\}$, respectively). Thus, $\traj(Z_{n+1})=1+\lfloor\frac{n}{2}\rfloor$. The result follows. 
\end{proof}

We describe a bijection $\phi_A$ that carries a Dyck path of size $n+1$ with all peaks at odd height to a Motzkin path of length $n$, which is due to J. Tirrell mentioned in \cite{Callan}. Let $\U,\HH,\DD$ denote an up step, a horizontal step and a down step, respectively.

Let $P=p_1p_2\cdots p_{2n+2}\in\A_{n+1}$. Removing $p_1$ and $p_{2n+2}$, the corresponding Motzkin path $\phi_A(P)=v_1\cdots v_n$ is defined by setting
\begin{equation} \label{eqn:V-step}
v_j=\begin{cases}
\U &\mbox{if $p_{2j}p_{2j+1}=\N\N$} \\
\HH &\mbox{if $p_{2j}p_{2j+1}=\E\N$} \\
\DD &\mbox{if $p_{2j}p_{2j+1}=\E\E$}
\end{cases}
\end{equation}
for all $j$. For example, the Dyck path $P\in\A_7$ on the left hand side of Figure \ref{fig:Dyck-odd-Motzkin} is carries to the Motzkin path $\phi_A(P)$ on the right hand side.

\begin{lem} \label{lem:parity-A} $\traj(P)\equiv\traj(Z_{n+1})\pmod 2$ if and only if $\phi_A(P)$ contains an even number of up steps.
\end{lem}

\begin{proof}
By Lemmas \ref{lem:peak-valley} and \ref{lem:valley-peak}, we have $\traj(P)\equiv\traj(Z_{n+1})\pmod 2$ if and only if $\area(P)\equiv\area(Z_{n+1})\pmod 2$.
Using (\ref{eqn:V-step}), we determine the contribution to $\area(P)$ by each step $v_j$ ($1\le j\le n$) as follows.
\begin{itemize}
\item $v_j=\U$. Then $p_{2j}p_{2j+1}=\N\N$, where $p_{2j}, p_{2j+1}$ are the left edges of two consecutive rows of unit squares above the line $y=x$. Thus, $v_j$ contributes an odd value to $\area(P)$.
\item $v_j=\HH$.  Then $p_{2j}p_{2j+1}=\E\N$, where $p_{2j+1}$ is the left edges of a row of an even number of unit squares above the line $y=x$. Thus, $v_j$ contributes an even value to $\area(P)$.
\item $v_j=\DD$.  Then $p_{2j}p_{2j+1}=\E\E$, where neither of $p_{2j},p_{2j+1}$ is the left edge of a row of unit squares. Thus, $v_j$ contributes nothing to $\area(P)$.
\end{itemize}
Since $\area(Z_{n+1})=0$ and the Motzkin path $\phi_A(Z_{n+1})$ contains no up step, we have $\area(P)\equiv\area(Z_{n+1})\pmod 2$ if and only if $\phi_A(P)$ contains an even number of up steps. The result follows. 
\end{proof}

Now, let $Z^*_n=\N(\N\E)^{n-1}\E$. Note that $Z^*_n$ is a Dyck path in $\B_n$ with $\area(Z^*_n)=n-1$. For example, the path represented by dashed lines on the left hand side of Figure \ref{fig:Dyck-even-Motzkin} is the Dyck path $Z^*_7$.

\begin{lem} \label{lem:K*_n} For $n\ge 2$, $\traj(Z^*_n)$ is odd if $n\equiv 0, 1\pmod 4$ and even if $n\equiv 2, 3\pmod 4$.
\end{lem}

\begin{proof} Let $T$ be the meander corresponding to the ordered pair $(Z^*_n,L_n)$. Note that the upper (lower, respectively) arch configuration of $T$ is the noncrossing matching $\{(1,2n)\}\cup\{(2,3), (4,5), \dots, (2n-2,2n-1)\}$ ($\{(1,2n), (2,2n-1), \dots, (n,n+1)\}$, respectively). Thus, we have $\traj(Z^*_n)=1+\lfloor\frac{n}{2}\rfloor$. The result follows. 
\end{proof}

We describe a bijection $\phi_B$ that carries a Dyck path of size $n$ with all peaks at even height to a Riordan path of length $n$ \cite{Callan}.  

Let $P=p_1p_2\cdots p_{2n}\in\B_n$. The corresponding Riordan path $\phi_B(P)=u_1\cdots u_n$ is defined by setting
\begin{equation} \label{eqn:U-step}
u_j=\begin{cases}
\U &\mbox{if $p_{2j-1}p_{2j}=\N\N$} \\
\HH &\mbox{if $p_{2j-1}p_{2j}=\E\N$} \\
\DD &\mbox{if $p_{2j-1}p_{2j}=\E\E$.}
\end{cases}
\end{equation}
for all $j\in [n]$. For example, the path represented by solid lines on the left hand side of Figure \ref{fig:Dyck-even-Motzkin} is carries to the Riordan path $\phi_B(P)$ on the right hand side.

\begin{figure}[ht]
\begin{center}
\includegraphics[width=3.8in]{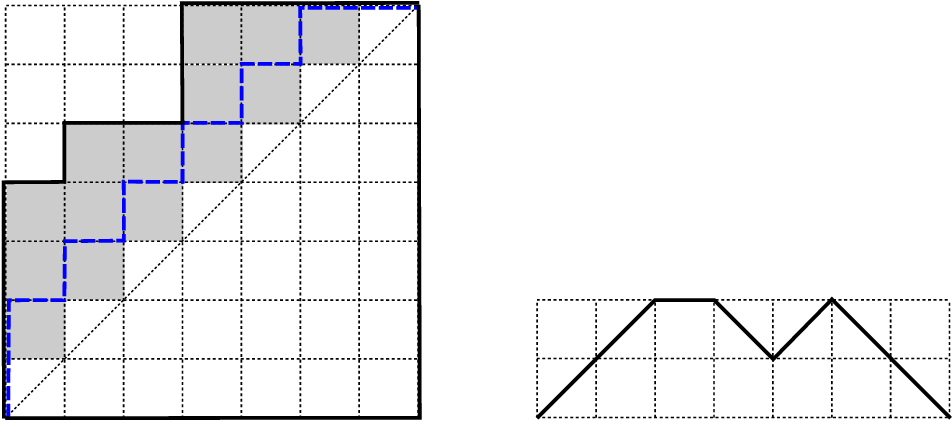}
\end{center}
\caption{\small A Dyck path $P\in\B_7$ and its corresponding Riordan path $\phi_B(P)$.} 
\label{fig:Dyck-even-Motzkin}
\end{figure}

\begin{lem} \label{lem:parity-B} $\traj(P)\equiv\traj(Z^*_n)\pmod 2$ if and only if $\phi_B(P)$ contains an odd number of up steps.
\end{lem}

\begin{proof}
By Lemmas \ref{lem:peak-valley} and \ref{lem:valley-peak}, we have $\traj(P)\equiv\traj(Z^*_n)\pmod 2$ if and only if $\area(P)\equiv\area(Z^*_n)\pmod 2$. Using (\ref{eqn:U-step}), we determine the contribution to $\area(P)-\area(Z^*_n)$ by each step $u_j$ ($2\le j\le n$) as follows.
\begin{itemize}
\item $u_j=\U$. Then $p_{2j-1}p_{2j}=\N\N$, where $p_{2j}, p_{2j+1}$ are the left edges of two consecutive rows of unit squares above the path $Z^*_n$. Thus, $u_j$ contributes an odd value to $\area(P)-\area(Z^*_n)$.
\item $u_j=\HH$. Then $p_{2j-1}p_{2j}=\E\N$, where $p_{2j}$ is the left edges of a row of an even number of unit squares above the path $Z^*_n$. Thus, $u_j$ contributes an even value to $\area(P)-\area(Z^*_n)$.
\item $v_j=\DD$ is a down step. Then $p_{2j-1}p_{2j}=\E\E$, where neither of $p_{2j-1},p_{2j}$ is the left edge of a row of unit squares. Thus, $u_j$ contributes nothing to $\area(P)-\area(Z^*_n)$.
\end{itemize}
Since the Riordan path $\phi_B(Z^*_n)$ contains exactly one up step, we have $\area(P)\equiv\area(Z^*_n)\pmod 2$ if and only if $\phi_B(P)$ contains an odd number of up steps. The result follows. 
\end{proof}

\noindent
\emph{Proof of Theorem \ref{thm:up-to-a-sign}.} Let $P\in A_{n+1}$.
Suppose the Motzkin path $\phi_A(P)$ contains an even number of up steps. By Lemmas \ref{lem:Z_n+1} and \ref{lem:parity-A}, $\traj(P)$ is odd if $n\equiv 0, 1\pmod 4$, and even if $n\equiv 2, 3\pmod 4$. On the other hand, suppose the Motzkin path $\phi_A(P)$ contains an odd number of up steps, it follows that $\traj(P)$ is even if $n\equiv 0, 1\pmod 4$, and odd if $n\equiv 2, 3\pmod 4$. Thus, the expression on the right hand side of (\ref{eqn:A_n+1}) is the excess of the number of Dyck paths $P\in A_{n+1}$ with even $\traj(P)$ over the ones with odd $\traj(P)$, which equals the left hand side of (\ref{eqn:A_n+1}).  The assertion (i) of Theorem \ref{thm:up-to-a-sign} follows. 

Let $P\in B_n$.
Suppose the Riordan path $\phi_B(P)$ contains an odd number of up steps. By Lemmas \ref{lem:K*_n} and \ref{lem:parity-B}, $\traj(P)$ is odd if $n\equiv 0, 1\pmod 2$, and even if $n\equiv 2, 3\pmod 2$. On the other hand, suppose the Riordan path $\phi_B(P)$ contains an even number of up steps, it follows that $\traj(P)$ is even if $n\equiv 0, 1\pmod 4$, and odd if $n\equiv 2, 3\pmod 4$. Thus, the expression on the right hand side of (\ref{eqn:B_n}) is the excess of the number of Dyck paths $P\in B_{n}$ with even $\traj(P)$ over the ones with odd $\traj(P)$, which equals the left hand side of (\ref{eqn:B_n}).  The assertion (ii) of Theorem \ref{thm:up-to-a-sign} follows.
\qed


\end{document}